\documentclass[11pt,a4paper]{amsart}
\usepackage{cancel}
\usepackage{amsmath}
\usepackage{amsthm}
\usepackage{amssymb}
\usepackage{euler}
\usepackage[left=1.25in, right=1.25in, bottom=1.5in]{geometry}

\newtheorem{thm}{Theorem}
\newtheorem{prop}{Proposition}

\newcommand\comaj{{\mathop\mathrm{comaj}}}

\DeclareFontFamily{U}{wncy}{}
\DeclareFontShape{U}{wncy}{m}{n}{<->wncyr10}{}
\DeclareSymbolFont{mcy}{U}{wncy}{m}{n}
\DeclareMathSymbol{\Sh}{\mathord}{mcy}{"58}

\newcommand\martial{{\Large \mars\!\!}}
\newcommand\tensor\otimes
\newcommand\qtorial{{{\textstyle q}\atop {\scriptscriptstyle\bullet}}}

\newcommand\calS{{\mathcal S}}
\newcommand\NN{{\mathbb N}}
\newcommand\ZZ{{\mathbb Z}}
\newcommand\QQ{{\mathbb Q}}

\newcommand\CC{{\mathbb C}}
\newcommand\defn[1]{{\bf #1}}
\newcommand\onto\twoheadrightarrow
\newcommand\into\hookrightarrow
\newcommand\iso\cong
\newcommand\from\longleftarrow

\newcommand\actson{\circlearrowright}
\newcommand\maj{{\mathop\mathrm{maj}}}

\usepackage{lipsum,wasysym, mathtools}

\begin{document}

\title[The commutant of divided difference operators]{The commutant of divided difference operators, Klyachko's genus,
  and the comaj statistic}

\author{Christian Gaetz}
\address{Department of Mathematics, University of California, Berkeley}
\email{gaetz@berkeley.edu}

\author{Rebecca Goldin}
\address{Department of Mathematical Sciences, George Mason University}
\email{rgoldin@gmu.edu}

\author{Allen Knutson}
\address{Department of Mathematics, Cornell University}
\email{allenk@math.cornell.edu}

\date{August 6, 2024}

\begin{abstract}
In \cite{ZachEtAl,Nenashev,AnnaOliver} are studied certain
  operators on polynomials and power series that commute with all
  divided difference operators $\partial_i$. We introduce a second set
  of ``martial'' operators $\martial_i$ that generate the full commutant,
  and show how a Hopf-algebraic approach naturally reproduces the operators
  $\xi^\nu$ from \cite{Nenashev}. We then pause to study Klyachko's homomorphism
  $H^*(Fl(n)) \to H^*($the permutahedral toric variety$)$, and extract the
  part of it relevant to Schubert calculus, the ``affine-linear genus''.
  This genus is then re-obtained using Leibniz combinations of the
  $\martial_i$. We use Nadeau-Tewari's $q$-analogue of Klyachko's genus
  to study the equidistribution of $\ell$ and $\comaj$
  on $[n]\choose k$, generalizing known results on $S_n$.
\end{abstract}


\maketitle

\section{The martial operators $\martial_\pi$}

\subsection{The ring of Schubert symbols}

Given a Dynkin diagram $D$ with Weyl group $W(D)$, define the
\defn{ring of Schubert symbols} $H(D)$ as the cohomology ring of
the associated (possibly infinite-dimensional) flag variety,
with the usual Schubert basis $\{\calS_w\colon w\in W(D)\}$. 
The Dynkin diagrams that will interest us are primarily the
semi-infinite $A_{\ZZ_+}$ and the biinfinite $A_\ZZ$.
In these type $A$ cases $W(D)$ is the group of finite permutations
of $\ZZ_+$ or of $\ZZ$. An important difference between the two is
that $H(A_{\ZZ_+})$ is generated by $\{\calS_{r_i}\colon i\in \ZZ_+\}$, where $r_i$ is a simple transposition,
so the multiplication is entirely determined by Monk's rule, whereas $H(A_{\ZZ})$
requires additional generators $\{\calS_{r_1r_2\cdots r_k}\}$ and determining
its multiplication involves also the flag Pieri rule.
With all that in mind we largely abandon the geometry and work with these
rings symbolically.

For each vertex $\alpha$ of $D$ hence generator $r_\alpha \in W(D)$,
we have an operator $\partial_\alpha \actson H(D)$,
pronounced ``partial $\alpha$'':
$$ \partial_\alpha \,\calS_\pi :=
\begin{cases}
  \calS_{ \pi r_\alpha} &\text{if }  \pi r_\alpha < \pi \\
  0 &\text{if } \pi  r_\alpha > \pi
\end{cases}
$$
from which we can well-define $\partial_{\pi}$ for any $\pi \in W(D)$ using products.

\begin{thm}[Lascoux-Sch\"utzenberger]\label{thm:LS}
  There is an isomorphism $H(A_{\ZZ_+}) \to \ZZ[x_1,x_2,\ldots]$
  taking the Schubert symbol $\calS_\pi$ to its ``Schubert polynomial''
  $S_\pi(x_1,x_2,\ldots)$. On the target ring, $\partial_\alpha$ acts
  by Newton's divided difference operation.
\end{thm}

Call a ring homomorphism from $H(D)$ to some other ring a
\defn{genus},\footnote{This terminology is stolen from the study of
  various cobordism rings of a point, e.g. the ``Hirzebruch genus''
  and ``Witten genus'' are ring homomorphisms to $\ZZ$.}
making the above isomorphism the \defn{Lascoux-Sch\"utzenberger genus}.

It was observed in \cite{ZachEtAl,AnnaOliver} that the operator
$\nabla := \sum_i \frac{d}{dx_i}$ on the target ring has two nice properties:
 it commutes with each $\partial_i$, and
 its application to any Schubert polynomial is a positive combination
  of Schubert polynomials.
Our goal in this section is to characterize operations of the
first type, with an eye toward the second.
To study this commutant it will be handy to work with algebra actions.

\subsection{Two commuting actions of the nil Hecke algebra}
\label{ssec:twoactions}

Let $Nil(D)$ denote formal (potentially infinite) linear combinations of
elements $\{d_\pi \colon \pi \in W(D)\}$, with a multiplication defined by
$ d_\pi d_\rho :=
\begin{cases}
  d_{\pi \rho} &\text{if }\ell(\pi \rho) = \ell(\pi) + \ell(\rho) \\
  0 &\text{if }\ell(\pi \rho) < \ell(\pi) + \ell(\rho).
\end{cases}
$
This multiplication extends to infinite sums in a well-defined way, insofar
as any $w\in W(D)$ has only finitely many length-additive
factorizations. Slightly abusing\footnote{One ordinarily considers only
  finite linear combinations, but we have need of certain infinite ones,
  and this simplifies the statement of Theorem \ref{thm:commutant}.}
terminology, we call this $Nil(D)$ the \defn{nil Hecke algebra}.
The association $d_\pi \mapsto \partial_\pi$ gives an action
of the opposite algebra $Nil(D)^{op}$ on $H(D)$; the infinitude
of the sums in $Nil(D)$ is again not problematic, because $H(D)$'s elements
are finite sums of Schubert symbols.

Define $\martial_\alpha$ (pronounced ``martial $\alpha$'') by
$$ \martial_\alpha \calS_\pi :=
\begin{cases}
  \calS_{r_\alpha \pi} &\text{if } r_\alpha \pi < \pi \\
  0 &\text{if } r_\alpha \pi > \pi 
\end{cases}
$$ \hskip -.05in 
We can well-define $\martial_{\prod Q} := \prod_{q\in Q} \martial_q$
for each reduced word $Q$.
  
\begin{thm}\label{thm:commutant}
  The map $d_\pi \mapsto \martial_\pi$ defines an action of $Nil(D)$
  on $H(D)$ (unspoiled by the potential infinitude), commuting with
  the $Nil(D)^{op}$-action by the operators $\partial_\alpha$.
  Conversely, each operator on $H(D)$ that commutes with all operators
  $\partial_\alpha$ arises as the action of a unique element of $Nil(D)$.

  In short, $Nil(D)$ and $Nil(D)^{op}$ are one another's commutants
  in their actions on $H(D)$.
\end{thm}

This then characterizes the operators that commute with all $\partial_i$;
we don't know any significance of the resulting algebra again being $Nil(D)$.   

\begin{proof}
  For $h \in H(D)$, let $\int h$ denote the coefficient of $\calS_e$ in $h$.
  Each $h = \sum_\pi h_\pi \calS_\pi$ is determined by the values
  $$ \int (\partial_\rho h) = \int \sum_\pi h_\pi \partial_\rho \calS_\pi
  = \sum_\pi h_\pi \int \partial_\rho \calS_\pi
  = \sum_\pi h_\pi \delta_{\pi,\rho} = h_\rho 
$$ \hskip -.05in 
  We'll make use of the easy fact
  $ \int \martial_\pi \,\partial_\rho \,\calS_\sigma =
  \begin{cases}
    1 &\text{if $\sigma= \pi^{-1} \rho$ and }\ell(\sigma)=\ell(\pi)+\ell(\rho)\\
    0 &\text{otherwise.}
  \end{cases}
  $
  
  Now let $C$ be an operator on $H(D)$ commuting with all $\partial_\pi$.
  For each $\pi\in W(D)$, let $c_\pi := \int C(\calS_{\pi^{-1}})$.
  We then confirm $C = \sum_\pi c_\pi \martial_\pi$ using the
  determination above:
  \begin{eqnarray*}
    \int \partial_\rho C(\calS_\sigma)
    &=& \int C(\partial_\rho \calS_\sigma)
        = \int C\left(\calS_{\sigma \rho^{-1}}\, 
        \left[\ell(\sigma \rho^{-1})=\ell(\sigma)-\ell(\rho)\right] \right) \\
    &=& \left[\ell(\sigma \rho^{-1})=\ell(\sigma)-\ell(\rho)\right]\ 
        c_{\rho \sigma^{-1}} \\
    \int  \partial_\rho \left( \sum_\pi c_\pi \martial_\pi\right) (\calS_\sigma)
    &=&  \sum_\pi c_\pi \int \martial_\pi \,\partial_\rho \,\calS_\sigma 
    \ \ =\ \ \sum_\pi c_\pi \, \left[\sigma 
= \pi^{-1}\rho\right]\, \big[\ell(\sigma)=\ell(\pi)+\ell(\rho)\big] \\
    &=& c_{\rho \sigma^{-1}} \left[\ell(\rho\sigma^{-1})=\ell(\sigma)-\ell(\rho)\right]
  \end{eqnarray*}
  Here $[P]=1$ if $P$ is true, $[P]=0$ if $P$ is false, for a statement $P$.
\end{proof}

{\em Example.} The action of $\nabla := \sum_i \frac{d}{dx_i}$ on polynomials,
pulled back to an action on $H(A_{\ZZ_+})$, is given by the operator
$\sum_{n \in \NN_+} n \martial_n$.
What is particularly special about $\nabla$ is that it is
a differential (i.e. satisfies the Leibniz rule), and is of degree $-1$.

\begin{thm}\label{thm:Leibniz}
  Let $\sum_\alpha c_\alpha \martial_\alpha \in Nil(D)$ be an operator of
  degree $-1$. If it is a differential (and $D$ is simply-laced,
  for convenience) then each $c_\alpha$ is $\frac 1 2 \sum_\beta c_\beta$
  where the $\beta$s are $\alpha$'s neighbors in $D$.

  In particular if $D$ is of finite type $ADE$, the only system
  of coefficients $(c_\alpha)$ is zero. If $D=A_{\ZZ_+}$, the only options
  are multiples of $c_i \equiv i$. If $D=A_\ZZ$, the space of such systems
  is two-dimensional, spanned by $c_i \equiv i$ and $c_i \equiv 1$.
\end{thm}

Hence the $\nabla$ discovered in \cite{ZachEtAl} in the $A_{\ZZ_+}$ case was
the only such operator available. In \cite{Nenashev} it is explained
that $\xi=\sum_{i \in \ZZ} \martial_i$ is special to the back-stable situation of $A_\ZZ$; here
we see that it is the only new option. (The result
\cite[Theorem 6]{Nenashev} is very similar.)

\begin{proof}[Proof sketch]
  The proof amounts to applying $\sum_\alpha c_\alpha \martial_\alpha$ to
  $(\calS_{r_\alpha})^2 = \sum_\beta \calS_{r_\beta r_\alpha}$
  (computed using the Chevalley-Monk rule).
\end{proof}


\section{Not-quite-Hopf algebras and Nenashev operators}

\subsection{The dual algebras}

Define a pairing $Nil(D) \tensor_\ZZ H(D) \rightarrow \ZZ$ by
$$
 p \tensor s \mapsto \text{coefficient of $\calS_e$ in }p(s)
$$
and from there a map $Nil(D) \to H(D)^* := Hom_\ZZ(H(D),\ZZ)$.

The following is well-known to the experts, if not usually expressed
exactly this way (see e.g. \cite{BerensteinRichmond},
\cite[\S7.2]{LLS}). 

\begin{thm}\label{thm:Hopf}
  This map $Nil(D) \to H(D)^*$ is an isomorphism. Unfortunately the induced
  comultiplication $H(D) \to H(D)\tensor H(D)$ is not a ring homomorphism
  (example below), so the two are not thereby dual Hopf algebras.
  (There is an alternative statement explored in \cite{LXZ}.)
\end{thm}

\newcommand\BS{BS}

There is an analogue of Theorem \ref{thm:LS} for $H(A_\ZZ)$, taking each
$\calS_\pi$ to its
\defn{back-stable Schubert function} $\BS_\pi$ invented by the third author
(and independently by Buch and by Lee), which were studied in
\cite{LLS,Nenashev}. Define a \defn{back-stable function}
$p \in \ZZ[[\ldots,x_{-1},x_0,x_1,x_2,\ldots]]$ to be a power series
\begin{itemize}
\item of finite degree, such that
\item $p$ depends only on the variables $\{x_k, k<N\}$ for some $N\gg 0$, and
\item for some $M\ll 0$, p is symmetric in the variables $\{x_i, i\leq M\}$.
\end{itemize}
One way (as appears in \cite{Nenashev}) to think of the ring of
back-stable functions is as the image of the injection
\begin{eqnarray*}
  Symm \tensor_\ZZ \ZZ[\ldots,x_{-1},x_0,x_1,\ldots] &\to&
\ZZ[[\ldots,x_0,\ldots]] / \langle 
        \text{elementary symmetric functions}\rangle \\
p \tensor q &\mapsto& p(\ldots,x_{-2},x_{-1},x_0) \, q
\end{eqnarray*}

For $\pi \in W(A_\ZZ)$ considered as a finite permutation of $\ZZ$,
and $shift_N(i) := i+N$, observe for $N\gg 0$ that $\pi[N] \coloneqq shift_N(i)\circ \pi\circ shift_{-N}(i)$
 is a permutation of $\ZZ$ that
leaves $-\NN$ in place, and thus has a well-defined Schubert
polynomial. Define the \defn{back-stable Schubert function}
$$ \BS_\pi := \lim_{N\to \infty} 
S_{\pi[N]}(x_{1-N},x_{2-N},\ldots) 
$$
where the limit is computed coefficient-wise (note that any single coefficient settles down to a constant value for all large enough $N$). 

\begin{thm}\cite[Theorem 3.5]{LLS}
  \label{thm:backstable}
  The back-stable Schubert functions lie in, and are a $\ZZ$-basis of,
  the ring of back-stable functions. 
\end{thm}

In this co\"ordinatization we can compute the comultiplication on $H(A_\ZZ)$
and bound its failure to be a ring homomorphism. Transposing the
multiplication from \S\ref{ssec:twoactions} of $d_\pi d_\rho$, we obtain
$\Delta(\BS_\sigma) = \sum \{BS_\pi \tensor BS_\rho\colon
\sigma = \pi \rho,\ \ell(\sigma) = \ell(\pi)+\ell(\rho)\}$.
Then, alas,
\newcommand\ttensor{{\!\!\>\otimes\!\>\!}}  
\begin{eqnarray*}
  \Delta(\BS_{r_2}^2) =   \Delta(\BS_{r_1r_2} + \BS_{r_3r_2})
                          &=& (\BS_{r_1r_2}\ttensor 1)
                          + (\BS_{r_1}\ttensor \BS_{r_2})
                          + (1\ttensor \BS_{r_1r_2}) \\
                          &&+ (\BS_{r_3r_2}\ttensor 1)
                          + (\BS_{r_3}\ttensor \BS_{r_2})
                             + (1\ttensor \BS_{r_3r_2})\\
\neq\Delta(\BS_{r_2}) \Delta(\BS_{r_2}) = (\BS_{r_2}\ttensor 1 + 1\ttensor\BS_{r_2})^2
                          &=& (\BS_{r_1r_2}\ttensor 1)
                          + (\BS_{r_3r_2}\ttensor 1)
                          + (\BS_{r_2}\ttensor \BS_{r_2}) \\
                          &&+ (\BS_{r_2}\ttensor \BS_{r_2})
                          + (1 \ttensor \BS_{r_1r_2})
                             + (1\ttensor \BS_{r_3r_2} )
\end{eqnarray*}
Luckily $\Delta(\BS_\pi \BS_{\rho[N]})
= \Delta(\BS_{\pi \circ (\rho[N])})
= \Delta(\BS_{\pi}) \Delta(\BS_{\rho[N]})$
for $N\gg 0$. Call this property \defn{``separated Hopfness''}, to be
used below.

\subsection{The Fomin-Greene--Nenashev operators $\xi^\nu$}

With these identifications, and the self-duality of the Hopf algebra
$Symm$ of symmetric functions, we can interpret some results
of Nenashev \cite{Nenashev}:
$$
\begin{matrix}
  H(A_\ZZ) &\xrightarrow{\sim}& \{\text{back-stable functions}\}
  &\xleftarrow{\sim}& Symm \tensor_\ZZ \ZZ[\ldots,x_{-1},x_0,x_1,\ldots]
  &\onto& Symm \\ \\
  Nil(A_\ZZ) & & & \from & & & Symm
\end{matrix}
$$

The map $\onto$ is the \defn{Stanley genus}:
it takes $\calS_\pi$ to its \defn{Stanley symmetric function}
$St_\pi = \sum_\lambda a_\pi^\lambda\, Schur_\lambda$. The lower map, its transpose,
takes $Schur_\lambda$ to $\sum_\pi a_\pi^\lambda\, d_\pi$. If we let this
operator act on $H(A_\ZZ)$ under the $d_\pi \mapsto \martial_\pi$ action, we get
the \defn{Fomin-Greene--Nenashev operator}
$\xi^\lambda := \sum_\pi a_\pi^\lambda \,\martial_\pi$ \cite{FominGreene,Nenashev}.
(See also the $j_\lambda$ operators in the ``Peterson subalgebra''
defined in \cite[\S9.3]{LLS}, which are a double version of the $\xi^\lambda$.)

\newcommand\mm{{\mathfrak m}}
Let $\mm$ denote the kernel of the map $H(A_\ZZ) \onto Symm$.
Using the separated Hopfness and the fact that
$\BS_{\pi[N]}-BS_\pi \in \mm$, one shows that each
$\Delta(pq)-\Delta(p)\Delta(q)$
(which serves as a measure of non-Hopfness)
lies in $\mm \tensor H(A_\ZZ) + H(A_\ZZ) \tensor \mm$. Hence the map 
$H(A_\ZZ)\to Symm$ factors through a map of Hopf algebras. Dually, the
transpose map is a Hopf map to a Hopf sub-bialgebra of $Nil(A_\ZZ)$.
In particular this Hopf map explains Nenashev's
formul\ae\ \cite[\S 4.4]{Nenashev}
$$ \xi^\lambda \xi^\mu = \sum_\nu c_{\lambda\mu}^\nu\, \xi^\nu
\qquad\qquad \xi^\nu(pq) = \sum_{\lambda,\mu} c_{\lambda\mu}^\nu \, \xi^\lambda(p)
\,\xi^\mu(q) $$

\subsection{Interlude (not used elsewhere):
  topological origin of the $\{BS_\pi\}$}

The stability property underlying Lascoux-Sch\"utzenberger's definition
of Schubert polynomials is the fact that each $S_\pi \in H^*(Fl(n))$
is the pullback $\iota_{n+1}^*(S_{\pi\oplus 1})$ along a map
$\iota_{n+1}: Fl(n) \into Fl(n+1)$ taking $(E_\bullet)$ to
$(F_\bullet\colon F_{i\leq n} = E_i\oplus 0,\ F_{n+1} = E_n \oplus \CC)$.
Chaining these together, one builds an element of the inverse limit
of the cohomology rings, a ring
$ \ZZ[[x_1,x_2,\ldots]]/\langle$elementary symmetric functions $e_i\rangle$.
It was then Lascoux-Sch\"utzenberger's pleasant surprise that these
``inverse limit Schubert classes'' lie in (and exactly span) the image of the
injective ring homomorphism $\ZZ[x_1,x_2,\ldots]$ into this algebra.

This admits of a parallel story, based on a different map
$\iota_{1+2n+1}: Fl(2n) \into Fl(2n+2)$ taking $(E_\bullet)$ to
$(F_\bullet\colon F_{i\in [1,2n+1]} = \CC \oplus E_{i-1}\oplus 0,\
F_{2n+2} = \CC \oplus \CC^{2n} \oplus \CC)$. Now, in order to achieve a
coherent labeling (as $n$ varies)
we index the classes in $H^*(Fl(2n))$ using permutations
of $[1-n, n]$ rather than of $[1,2n]$. Once again the inverse limit is
a power series ring modulo elementary symmetrics, but it is {\em no
  longer true} that the inverse limit Schubert classes are representable
by polynomials; rather, they can be represented by back-stable functions.
(And again, they form a basis thereof.)

One advantage of $\iota_{1+2n+1}$ is that it is equivariant w.r.t. the
{\em duality} endomorphism of $Fl(2n)$, which takes $(E_\bullet)$ to
$(E_\bullet^\perp)$,
defined w.r.t. the symplectic form pairing co\"ordinates $i$ and $1-i$, for
$i\in [1,n]$. On the level of classes, this
takes $BS_\pi \mapsto BS_{w_0 \pi w_0}$
where $w_0(i) := 1-i$. On the level of back-stable functions,
it takes $x_i \mapsto -x_{1-i}$,\ $e_i(x_{\leq 0}) \mapsto e_i(x_{\leq 0})$.

Since this duality respects Schubert classes and the alphabet $(x_i)$,
it takes Monk's rules to Monk's rules. In particular it turns the
transition formula (a specific Monk's rule)
$$ BS_\pi = x_i BS_{\pi'} + \sum_{\text{certain }\pi''} BS_{\pi''}
\qquad\qquad\text{into}\qquad\qquad
BS_\rho = -x_j BS_{\rho'} + \sum_{\text{certain }\rho''} BS_{\rho''} $$
\hskip -.05in
which implies (unstably) the cotransition formula
$x_j \calS_{\rho'} = -\cancel{\calS_\rho}
+ \sum_{\text{certain }\rho''} \calS_{\rho''} $
of \cite{K}.

\section{Relation to Klyachko's genus}

\subsection{Klyachko's ideal and its prime factors}

Let $T \leq GL_n(\CC)$ denote the group of diagonal matrices, and
$TV_{perm} \subseteq Fl(n)$ be the \defn{permutahedral toric variety}
obtained as the closure of a generic $T$-orbit on the flag manifold $Fl(n)$.
This subvariety arises as a Hessenberg variety (see e.g. \cite{abe2020survey})
and is of key importance in \cite{Huh,NadeauTewari}. 

The inclusion $\iota\colon TV_{perm} \into Fl(n)$ induces a map
backwards on cohomology, which is neither injective nor surjective.
Klyachko \cite{klyachko1985orbits} presented its image
$im(\iota^*)$ (with rational coefficients), and a formula for $\iota^*$
evaluated on Schubert symbols:
\begin{eqnarray*}
  H^*(Fl(n);\QQ) &\to& im(\iota^*) \iso \QQ[k_0,\ldots,k_n]\bigg/
                   \left\langle
                   \begin{array}{c}
                     k_i(-k_{i-1}+2k_i-k_{i+1}) = 0,\ \  1\leq i \leq n-1 \\
                   k_0=k_n=0 
                   \end{array}\right\rangle \\                     
  \calS_\pi &\mapsto& \frac{1}{\ell(\pi)!} \sum_{Q \in RW(\pi)} \ \prod_{q\in Q} k_q
\quad\text{where $RW(\pi)$ is the set of reduced words}
\end{eqnarray*}

Taking forward- and back-stable limits, while leaving behind geometry,
we get the \defn{Klyachko genus}
$
  H(A_\ZZ) \to \QQ[\ldots,k_{-1},k_0,k_1,\ldots]\bigg/
                   \left\langle \
                     k_i(-k_{i-1}+2k_i-k_{i+1}) = 0\ \forall i\in \ZZ \
                   \right\rangle 
$
whose map on Schubert symbols is given by the same formula.
We use this to recover a result of Nenashev,
foreshadowing some results in \S\ref{sec:egg}:

\begin{thm}\cite[Proposition 3 and discussion after]{Nenashev}
\label{thm:Nenashev}
  Let $RW(\pi)$ denote the set of reduced words for $\pi$. There must
  {\bf exist} (but the proof doesn't find one) a ``rectification'' map
  $$ \{\text{shuffles of any word in $RW(\pi)$ with any word in $RW(\rho)$}\}
  \to \coprod_\sigma RW(\sigma) 
$$ \hskip -.05in 
  whose fiber over any reduced word for $\sigma$ has
  size 
  $c_{\pi\rho}^\sigma$, the coefficient from
$ 
    \calS_\pi \calS_\rho = \sum_\sigma c_{\pi\rho}^\sigma \, \calS_\sigma.
$ 
\end{thm}

\begin{proof}
  Apply the Klyachko genus to that last equation, 
  then set all $k_i=1$, obtaining
  $$
  \frac{1}{\ell(\pi)!}\sum_{P\in RW(\pi)}\prod_P 1 \ \ \ \
  \frac{1}{\ell(\rho)!}\sum_{R\in RW(\rho)}\prod_R 1
  =
  \sum_\sigma c_{\pi\rho}^\sigma \,
  \frac{1}{\ell(\sigma)!}\sum_{S\in RW(\sigma)}\prod_S 1
  $$
  Since $c_{\pi\rho}^\sigma = 0$ unless $\ell(\sigma) = \ell(\pi)+\ell(\rho)$,
  we can restrict to those $\sigma$. Multiplying through:
  $$
  \#RW(\pi)\ \#RW(\rho) {\ell(\pi)+\ell(\rho)\choose \ell(\pi)}
  =
  \sum_\sigma c_{\pi\rho}^\sigma \, \#RW(\sigma)
  $$ \hskip -.05in 
  Let $C_{\pi\rho}^\sigma$ be a set with cardinality $c_{\pi\rho}^\sigma$
  (and wouldn't you like to know one?). Then we can interpret the above as
  $$ \#
   \{\text{shuffles of any word in $RW(\pi)$ with any word in $RW(\rho)$}\}
   \ =\ \# \coprod_\sigma (C_{\pi\rho}^\sigma \times RW(\sigma)) 
$$ \hskip -.05in 
   Hence there exists a bijection; compose it with the
   projection to $\coprod_\sigma RW(\sigma)$.
\end{proof}

We can further simplify the target of this genus by modding
out by each of the minimal prime ideals that contain the Klyachko
ideal. We get ahold of these using the Nullstellensatz,\footnote{This
  isn't quite fair, in that we are working in infinite dimensions, but
  we won't worry about it. All we're really trying to do here is choose,
  for each $i$, which factor of $k_i(-k_{i-1}+2k_i-k_{i+1})$
  to mod out.}  i.e. by looking at the components of the
solution set to Klyachko's equations.

\begin{prop}\label{prop:Nullstellensatz}
  Consider $\ZZ$-ary tuples $(k_i)_{i\in \ZZ}$ of complex numbers satisfying
  the Klyachko equalities. This set is the (nondisjoint) union of the
  following countable set of $2$-planes:
  \begin{itemize}
  \item For $a,b \in \CC$,  let $k_m = am+b$.  
  \item For $i\leq j$ each in $\ZZ$, and $x,y \in \CC$ a pair of
    ``slopes'', let
    $    k_m =    \begin{cases}
      x(m-i) & \text{if }k \leq i \\
      0 & \text{if }k \in [i,j] \\
      y(m-j) & \text{if }k \geq j.
    \end{cases}
    $
  \end{itemize}
\end{prop}

After completing this work, we learned of a very similar calculation
in \cite[\S 3.4]{NadeauTewari}, so we omit the proof of proposition
\ref{prop:Nullstellensatz} (obtainable as a sort of $q\to 1$ limit of theirs).

Each component defines a quotient of the Klyachko ring, namely
$$
\QQ[\ldots,k_{-1},k_0,k_1,\ldots]\bigg/
\left\langle
  -k_{m-1}+2k_m-k_{m+1} = 0\ \forall m\in \ZZ
\right\rangle
$$
$$
\forall i\leq j, \qquad
\QQ[\ldots,k_{-1},k_0,k_1,\ldots]\bigg/
\left\langle
  \begin{array}{cc}
    k_m = 0 &\forall m\in [i,j] \\
  -k_{m-1}+2k_m-k_{m+1} = 0& \forall m\notin [i,j]
  \end{array}
\right\rangle
$$ \hskip -.05in 
Call the map of $H(A_\ZZ)$ to the first quotient the \defn{affine-linear genus}.

There is a slight subtlety in that the Klyachko ideal is not radical,
and as such, the map from the Klyachko ring to the direct sum of these
quotients is not injective. We will return to this minor matter below.

\subsection{Dropping the other genera}

The other components (besides the one giving the affine-linear genus)
are useless, in the following senses.  Say $k_m = 0$ for some $m$;
then there are three situations.
\begin{enumerate}
\item {\em Some reduced word for a permutation $\pi$ uses the letter  $m$.}
  Then {\em all} reduced words do, with the effect that $\calS_\pi \mapsto 0$
  in the quotient ring.
\item {\em Each reduced word for $\pi$ uses some letters $>m$ and some $<m$.}
  Then $\pi = \pi_{<m} \pi_{>m}$ where each
  uses only letters $<m$, $>m$ respectively. In this case
  $\calS_\pi = \calS_{\pi_{<m}} \calS_{\pi_{>m}}$.
\item {\em Each reduced word for $\pi$ only uses letters on one side of $m$.}
  At this point there is nothing to be gained by setting $k_m=0$; we could
  work with just the affine-linear genus.
\end{enumerate}
Our principal interest in genera is to study \defn{Schubert calculus},
the structure constants $c_{\pi\rho}^\sigma$
of the multiplication of Schubert symbols.
That is hard to do if the symbols map to zero (situation \#1), silly to do
directly if the symbols are are themselves products (situation \#2),
and in situation \#3 might as well be done using the affine-linear genus.
As such, at this point we cast aside the Klyachko genus in favor
of the affine-linear genus $\gamma$:
\begin{eqnarray*}
  \gamma:\ H(A_\ZZ) &\to& \QQ[a,b], \qquad \qquad\qquad
  \calS_\pi \mapsto \frac{1}{\ell(\pi)!}\ \sum_{P \in RW(\pi)}\ \prod_{i\in P} (ai+b)
\end{eqnarray*}

The assiduous reader might be guessing now that the information lost
when passing from the Klyachko ideal to its radical is similarly
negligible for Schubert calculus purposes. And indeed: if we factor the
Klyachko ideal as an intersection of primary instead of prime components,
we run into the ideals
$$
\forall i\leq j, \qquad
\QQ[\ldots,k_{-1},k_0,k_1,\ldots]\bigg/
\left\langle
  \begin{array}{cc}
    k_m^2 = 0 &\forall m\in [i+1,j-1] \\
    k_i = k_j = 0 &\\
  -k_{m-1}+2k_m-k_{m+1} = 0& \forall m\notin [i,j]
  \end{array}
\right\rangle
$$
\hskip -.05in 
These would let us study $\pi,\rho,\sigma$ whose reduced words use only
the letters in the range $[i+1,j-1]$, and each at most once. This is an
extremely limited case.

\section{The affine-linear genus $\gamma$
  from the martial derivations}

Recall the derivations
$$ \nabla = \sum_m m \martial_m \qquad\qquad \xi = \sum_m \martial_m $$
\hskip -.05in 
{\em Being} derivations, they exponentiate to automorphisms of
$\QQ \tensor_\ZZ H(A_\ZZ)$ (where the $\QQ$ is necessitated by
the denominators in the exponential series).

\begin{thm}\label{thm:triangle}
  The following triangle commutes: \quad
  $
  \begin{matrix}
    && H(A_\ZZ) \\ 
    & \qquad\qquad e^{a\nabla+b\xi} \swarrow  && \searrow \gamma \quad \quad \\ 
    & \QQ[a,b] \tensor_\ZZ H(A_\ZZ) & \onto & \QQ[a,b] \\ 
    & \calS_{\pi} & \mapsto & \delta_{\pi,e}
  \end{matrix}
  $
\end{thm}

\begin{proof}
  The proof is not conceptual; we compute both sides and compare.
  Indeed, we find the statement intriguing exactly because we know of
  no geometric reason the two maps should be related.
  $$ e^{a\nabla+b\xi} \cdot \calS_\pi
  \ =\ \sum_n \frac{1}{n!} (a\nabla+b\xi)^n \cdot \calS_\pi
  \ \mapsto\ \frac{ (a\nabla+b\xi)^{\ell(\pi)} \cdot \calS_\pi }{\ell(\pi)!}
  \ =\ \frac{ \left(\sum_i (ai+b)\martial_i\right) ^{\ell(\pi)} \cdot \calS_\pi }{\ell(\pi)!}
  $$
  \hskip -.05in 
  Expanding $\left(\sum_i (ai+b)\martial_i\right) ^{\ell(\pi)}$,
  the nonvanishing terms correspond to reduced words of length $\ell(\pi)$,
  and only those that multiply to $\pi^{-1}$ survive application to $\calS_\pi$.
\end{proof}

In particular the proof of Theorem \ref{thm:Nenashev} essentially
amounts to applying $\exp(\xi)$. (Oddly, the original proof in \cite{Nenashev}
is closer to an application of $\exp(\nabla)$.)

There is a fascinating \defn{$q$-Klyachko genus} introduced in
\cite[\S3.4]{NadeauTewari}:
\begin{eqnarray*}
  \gamma_q:\ H(A_\ZZ) &\to& \QQ(q)[\alpha,\beta] \\
  \calS_\pi &\mapsto& \frac{1}{\ell(\pi)\qtorial} \sum_{Q:\ \prod Q = \pi} \
                      q^{\comaj(Q)}  \prod_{i\in Q} \left(\alpha q^i+\beta\right)
\end{eqnarray*}
Here 
$m\qtorial$ is the $q$-torial $\prod_{j=1}^m [j]_q$, and $\comaj(Q)$ is the sum of
the positions of the ascents. We looked for a long time for a $q$-analogue of
Theorem \ref{thm:triangle}, to no avail: it would provide an
automorphism of $H(A_\ZZ)(q)[\alpha,\beta]$
whose $\ell=0$ part is the $q$-Klyachko genus. 

\section{Rectification and the $q$-statistic}\label{sec:egg}

We pursue a $q$-analogue of (Nenashev's) Theorem \ref{thm:Nenashev}.
Applying Nadeau-Tewari's $q$-Klyachko genus to 
$\calS_\pi \calS_\rho = \sum_\sigma c_{\pi\rho}^\sigma \, \calS_\sigma$
we get
\begin{eqnarray*}
&&  \frac{1}{\ell(\pi)\qtorial}
  \sum_{P\in RW(\pi)} \ q^{\comaj(P)}  \prod_{i\in P} (\alpha q^i + \beta) \quad
\frac{1}{\ell(\rho)\qtorial}
  \sum_{R\in RW(\rho)} \ q^{\comaj(Q)}  \prod_{i\in R} (\alpha q^i + \beta) \\
  &=&
\sum_\sigma c_{\pi\rho}^\sigma \, 
  \frac{1}{\ell(\sigma)\qtorial}
  \sum_{S\in RW(\sigma)} \ q^{\comaj(S)}  \prod_{i\in S} (\alpha q^i + \beta)
\end{eqnarray*}
Multiplying through, we get
\begin{eqnarray*}
  {\ell(\pi)+\ell(\rho)\choose \ell(\pi)}_q\
  \sum_{P \in RW(\pi)\atop R\in RW(\rho)} q^{\comaj(P)\atop +\comaj(R)} \prod_{i\in P\coprod R}(\alpha q^i + \beta)
  &=&
\sum_\sigma c_{\pi\rho}^\sigma \, 
  \sum_{S\in RW(\sigma)} \ q^{\comaj(S)}  \prod_{i\in S} (\alpha q^i + \beta)
\end{eqnarray*}

Let's interpret both sides at $\alpha = \beta = q = 1$, again using
a mystery set $C_{\pi\rho}^\sigma$ with cardinality $c_{\pi\rho}^\sigma$.
Define a \defn{barred word} for $\pi$ as a reduced word in which some
letters are overlined, e.g.  $12\overline 1$ for $(13)$.
Then the left side of the above equation counts pairs $(P,R)$ of
barred words, shuffled together, where the barring indicates ``use the
$\alpha q^i$ term'' rather than the $\beta$ term.  Meanwhile, the right side
counts pairs $(\tau,S)$ where $S$ is a barred word for some $\sigma$,
and $\tau$ is in $C_{\pi\rho}^\sigma$.

\begin{thm}\label{thm:qNenashev}
  Define the $q$-statistic of a barred word as the sum of the locations
  of the ascents, plus the sum of the barred letters.
  
  Define the $q$-statistic of a shuffle $\Sh$ of a pair $(P,R)$
  of barred words as the sum of the two $q$-statistics, plus the number
  of inversions in the shuffle (letters in $R$ leftward of letters in $P$).

  Then there {\bf exists} (but the proof doesn't find one)
  a ``rectification'' map
  $$ \{\text{shuffles of pairs $(P,R)$ of barred words for $\pi,\rho$}\}
  \quad \to \quad \coprod_\sigma\ \{\text{barred words for $\sigma$}\}
  $$
  preserving the number of bars and the $q$-statistic, whose fiber
  over each word for $\sigma$ is of size $c_{\pi\rho}^\sigma$.
\end{thm}

We note that the affine-linear genus doesn't let one produce such a
combinatorial result, insofar as the factors $ai+b$ can involve $i<0$
(in the back-stable setting of $A_\ZZ$). 

{\em Example.} These examples get large very quickly, so we restrict to
the fully barred case. Let $\pi=\rho=12463578$, chosen to give a
$c_{\pi\rho}^\sigma>1$ (and chosen stably enough that the terms in the
product don't move $-\NN$). Each of $\pi$ and $\rho$ have two reduced
words ($354$ and $534$, $\comaj$s $1$ and $2$, each of total $12$),
and there are $6\choose 3$ ways to shuffle, for a total of
$2\cdot 2\cdot {6\choose 3} = 80$; the resulting $q$-statistics range
from $26 = 1+12+1+12+0$ to $37 = 2+12+2+12+3\cdot 3$.
There are $7$ terms $\calS_\sigma$ in the product
$\calS_\pi\calS_\rho$ (one with coefficient $2$) with various numbers
of reduced words.  \setcounter{MaxMatrixCols}{40}
$$
\begin{matrix}
  \text{$q$-statistic:} & 26 & 27 & 28 & 29 & 30 & 31 & 32 & 33 & 34 & 35 & 36 & 37\\ \hline
  & 1 & 3 & 5 & 8 & 11 & 12 & 12 & 11 & 8 & 5 & 3 & 1 &\text{total $=80$}\\
  \sigma\\
23561478 & 1 & 1 & 2 & 1 & 2 & 1 & 1 &   &   &   &   &  \\
14562378 &   & 1 &   & 1 & 1 & 1 &   & 1 &   &   &   &  \\
13572468 &   &   & 1 & 2 & 2 & 3 & 3 & 2 & 2 & 1 &   &  \\
13572468 &   &   & 1 & 2 & 2 & 3 & 3 & 2 & 2 & 1 &   &   &\text{again}\\
23471568 &   & 1 & 1 & 2 & 2 & 2 & 1 & 1 &   &   &   &  \\
13482567 &   &   &   &   & 1 & 1 & 2 & 2 & 2 & 1 & 1 &  \\
12673458 &   &   &   &   & 1 &   & 1 & 1 & 1 &   & 1 &  \\
12583467 &   &   &   &   &   & 1 & 1 & 2 & 1 & 2 & 1 & 1 
\end{matrix}
$$ \hskip -.05in 
In the line with ``total $= 80$'', we count the number of fully barred
shuffles with given $q$-statistic. In each of the lower lines,
we put $\sigma$ on the left, and on the right
we list the number of fully barred words for it with given $q$-statistic.
Then the theorem asserts that each number atop a column is the sum of
the numbers below. There is a silly rotational near-symmetry tracing to the
fact that $\pi$ and $\rho$ are Grassmannian permutations for
self-conjugate partitions.

\section{Equidistribution of inversion number
  vs. $\comaj$ on $[n]\choose m$}

Let $J \subseteq \ZZ$ be a set of $n$ numbers,
no two adjacent. Then the product $\prod_{j\in J} r_j$ is well-defined i.e$.$
is independent of the order; indeed, the reduced words for $\prod_{j\in J} r_j$
are in correspondence with permutations of $J$.
The same holds when multiplying subsets of $J$.

Fix $K\subseteq J$ and let $\rho = \prod_K r_k$,
$\pi = \prod_{J\setminus K} r_j$. 
Then $\calS_\pi \calS_\rho = \calS_{\pi \rho}$, and Theorem \ref{thm:qNenashev}
(again in the fully barred case) predicts a bijection
$$ \{\text{insertions of reduced words $R$ for $\rho$
  into reduced words $P$ for $\pi$} \}  \quad
\to \quad RW(\pi \rho)
$$
\hskip -.05in 
such that $\big[\comaj(P)$ + $\comaj(R)$ + the inversion number of the
shuffle$\big]$
matches $\comaj$ of the resulting word $\Sh$. Note that the obvious map (just
insert $R$ where the shuffle suggests) does {\em not} correspond
these two statistics!

If we break $J$ not into two subsets, but all the way down into individual
letters, this recovers the equidistribution on $S_n$ of the statistics
$\ell$ and $\comaj$ (or $\maj$); see e.g. \cite[Proposition 1.4.6]{EC1}.

This hints at a stronger result: that for any two strings $P,R$ such that $PR$ has no repeats,
on the set $\{$shuffles $\Sh\}$ the distributions of the statistic
$\comaj(P) + \comaj(R) + \ell(\Sh)$ and the statistic $\comaj(\Sh)$ match.
(Theorem \ref{thm:qNenashev} only guarantees this {\em after summing}
over all $P\in RW(\pi),R\in RW(\rho)$.) 
And indeed, this stronger claim holds \cite{GarsiaGessel}.

\section*{Acknowledgements}
We thank Marcelo Aguiar, Hugh Dennin, Yibo Gao, Thomas Lam,
  Philippe Nadeau, Gleb Nenashev, Mario Sanchez, Thomas B\r a\r ath
  Sj\"oblom, and David E Speyer. RG was partially supported by NSF grant DMS--2152312. AK was partially supported by NSF grant DMS--2246959.

\bibliographystyle{plain}
\bibliography{martials}

\begin{thebibliography}{10}

\bibitem{abe2020survey}
Hiraku Abe and Tatsuya Horiguchi.
\newblock A survey of recent developments on {H}essenberg varieties.
\newblock {\em \\ Schubert Calculus and Its Applications in Combinatorics and
  Representation Theory: \\ Guangzhou, China, November 2017}, pages 251--279,
  2020.

\bibitem{BerensteinRichmond}
Arkady Berenstein and Edward Richmond.
\newblock Littlewood--{R}ichardson coefficients for reflection groups.
\newblock {\em Advances in Mathematics}, 284:54--111, 2015.

\bibitem{FominGreene}
Sergey Fomin and Curtis Greene.
\newblock Noncommutative {S}chur functions and their applications.
\newblock {\em Discrete Mathematics}, 193(1-3):179--200, 1998.

\bibitem{GarsiaGessel}
A.~M. Garsia and I.~Gessel.
\newblock Permutation statistics and partitions.
\newblock {\em Adv. in Math.}, 31(3):288--305, 1979.

\bibitem{ZachEtAl}
Zachary Hamaker, Oliver Pechenik, David~E Speyer, and Anna Weigandt.
\newblock Derivatives of {S}chubert polynomials and proof of a determinant
  conjecture of {S}tanley.
\newblock {\em Alg. Comb.}, 3(2):301--307, 2020.

\bibitem{Huh}
June Huh and Eric Katz.
\newblock Log-concavity of characteristic polynomials and the {B}ergman fan of
  matroids.
\newblock {\em Mathematische Annalen}, 354:1103--1116, 2012.

\bibitem{klyachko1985orbits}
Aleksandr~Anatol’evich Klyachko.
\newblock Orbits of a maximal torus on a flag space.
\newblock {\em \\ Functional analysis and its applications}, 19(1):65--66,
  1985.

\bibitem{K}
Allen Knutson.
\newblock Schubert polynomials, pipe dreams, equivariant classes, and a
  co-transition formula.
\newblock In {\em Facets of algebraic geometry. {V}ol. {II}}, volume 473 of
  {\em London Math. Soc. Lecture Note Ser.}, pages 63--83. Cambridge Univ.
  Press, Cambridge, 2022.

\bibitem{LLS}
Thomas Lam, Seung~Jin Lee, and Mark Shimozono.
\newblock Back stable {S}chubert calculus.
\newblock {\em \\ Compositio Mathematica}, 157(5):883--962, 2021.

\bibitem{LXZ}
Martina Lanini, Rui Xiong, and Kirill Zainoulline.
\newblock Structure algebras, {H}opf algebroids and oriented cohomology of a
  group.
\newblock 2023.
\newblock arxiv:2303.02409.

\bibitem{NadeauTewari}
Philippe Nadeau and Vasu Tewari.
\newblock A $q$-deformation of an algebra of {K}lyachko and {M}acdonald's
  reduced word formula, 2021.
\newblock arxiv:2106.03828.

\bibitem{Nenashev}
Gleb Nenashev.
\newblock Differential operators on {S}chur and {S}chubert polynomials.
\newblock arxiv:2005.08329.

\bibitem{AnnaOliver}
Oliver Pechenik and Anna Weigandt.
\newblock A dual {L}ittlewood--{R}ichardson rule and extensions, 2022.
\newblock arxiv:2202.11185.

\bibitem{EC1}
Richard~P Stanley.
\newblock Enumerative {C}ombinatorics volume 1, second edition.
\newblock {\em \\ Cambridge Studies in Advanced Mathematics}, 2011.

\end{thebibliography}

\end{document}